\documentclass{amsart}
\usepackage{amsmath}
\usepackage{amsthm}
\usepackage{amsfonts}
\usepackage{dsfont}
\usepackage{color}
\usepackage{enumerate}
\usepackage{amssymb}
\usepackage{booktabs}
\usepackage{arydshln}
\usepackage{blkarray}
\usepackage{hhline}
\newcommand{\complex}{\mathbb{C}}

\newcommand{\para}[1]{\left(#1\right)}
\newcommand{\paraa}[1]{\big(#1\big)}

\newcommand{\spacearound}[1]{\quad#1\quad}
\newcommand{\equivalent}{\spacearound{\Leftrightarrow}}
\renewcommand{\implies}{\spacearound{\Rightarrow}}

\newtheorem{theorem}{Theorem}[section]

\newtheorem{lemma}[theorem]{Lemma}
\newtheorem{proposition}[theorem]{Proposition}
\newtheorem{example}[theorem]{Example}
\theoremstyle{definition}
\newtheorem{definition}[theorem]{Definition}
\theoremstyle{remark}

\numberwithin{equation}{section}

\title{On moduli spaces of K\"ahler-Poisson algebras over rational
  functions in two variables}% \\ \texttt{\small\today\, }}

\author{Ahmed Al-Shujary}
\address[Ahmed Al-Shujary]{Dept. of Math.\\
Link\"oping University\\
581 83 Link\"oping\\
Sweden}
\email{ahmed.al-shujary@liu.se}

\thanks{}


\begin{document}

\maketitle
\begin{abstract}
  K\"ahler-Poisson algebras were introduced as algebraic analogues of
  function algebras on K\"ahler manifolds, and it turns out that one
  can develop geometry for these algebras in a purely algebraic way. A
  K\"ahler-Poisson algebra consists of a Poisson algebra together with
  the choice of a metric structure, and a natural question arises: For
  a given Poisson algebra, how many different metric structures are
  there, such that the resulting K\"ahler-Poisson algebras are
  non-isomorphic? In this paper we initiate a study of such moduli
  spaces of K\"ahler-Poisson algebras defined over rational functions
  in two variables.
\end{abstract}

\section{Introduction}

\noindent
In \cite{algebras} we initiated the study of K\"ahler-Poisson algebras
as algebraic analogues of algebras of functions on K\"ahler
manifolds.  K\"ahler-Poisson algebras consist of a Poisson algebra
together with a metric structure. This study was motivated by the
results in \cite{Pseudo-Riemannian,multi-linear}, where many aspects
of the differential geometry of an embedded almost K\"ahler manifold
$\Sigma$ can be formulated in terms of the Poisson structure of the
algebra of functions of $\Sigma$. In \cite{algebras} we showed that
``the K\"ahler--Poisson condition'', being the crucial identity in the
definition of K\"ahler-Poisson algebras, allows for an identification
of geometric objects in the Poisson algebra which share important
properties with their classical counterparts. For instance, we proved
the existence of a unique Levi-Civita connection on the module
generated by the inner derivations of the K\"ahler-Poisson algebra,
and that the curvature operator has all the classical symmetries. In
\cite{morphism} we explore further algebraic properties of
K\"ahler-Poisson algebras. In particular, we find appropriate
definitions of morphisms of K\"ahler-Poisson algebras as well as
subalgebras, direct sums and tensor products.

Starting from a Poisson algebra $\mathcal{A}$, it is interesting to
ask the following question: How many non-isomorphic K\"ahler-Poisson
algebras can one construct from $\mathcal{A}$?  This amounts to the
study of a ``moduli space'' for K\"ahler-Poisson algebras, in analogy
with the corresponding problem for Riemannian manifolds, where one
consider metrics giving rise to non-isometric Riemannian manifolds. In
general, this is a hard problem and, in this note, we will focus on
approaching this problem in a particular setting. To this end, we
start from the polynomial algebra $\complex[x,y]$. In order to
understand isomorphism classes of K\"ahler-Poisson algebras based on
this algebra, one needs to study automorphisms of $\complex[x,y]$. In
\cite{uber ganze} Jung shows that every automorphism of
$\complex[x,y]$ is a composition of so called elementary
automorphisms. This result was later extended to fields of arbitrary
characteristics and free associative algebras
(\cite{Czerniakiewicz,Dicks,Makar,Nagata,van der}). In the notation of
\cite{automorphism} (which also provides an elementary proof of the
automorphism theorem), every $k$-algebra automorphism of $k[x,y]$ is a
finite composition of automorphisms of the type:
\begin{enumerate}
\item $x\mapsto x $, $y\mapsto y+h(x)$ with $h(x)\in k[x]$ 
\item $x\mapsto a_{11}x+a_{12}y+a_{13}$,
  $y\mapsto a_{21}x+a_{22}y+a_{23}$ with
  $a_{11}a_{22}\neq a_{21}a_{12}$
\end{enumerate}
for $a_{ij}\in k$. Using these automorphisms, we shall initiate a
study of isomorphism classes of K\"ahler-Poisson algebras over
$\complex[x,y]$. However, as shown in \cite{algebras}, the
construction of a K\"ahler-Poisson algebra over $\complex[x,y]$ will
in most cases involve a localization of the algebra. Therefore, we
shall rather start from $\complex(x,y)$, the algebra of rational
functions of two variables, together with an appropriate linear
Poisson structure. Although we do not solve the problem in its full
generality, the classification results for certain classes of metrics
obtained below, give an insight into the complexity of the general
problem.

\section{K\"ahler-poisson algebras}

\noindent
We begin this section by recalling the main object of our
investigation. Let us consider a Poisson algebra
$(\mathcal{A},\{\cdot,\cdot\})$ over $\mathbb{C}$ and let
$\{x^1,...,x^m\}$ be a set of distinguished elements of
$\mathcal{A}$. These elements play the role of functions providing an
isometric embedding for K\"ahler manifolds
(cf. \cite{multi-linear}). Let us recall the definition of
K\"ahler-Poisson algebras together with a few basic results
(cf. \cite{algebras}).

\begin{definition} \label{def1}
  Let $(\mathcal{A},\{\cdot,\cdot\})$ be a Poisson algebra over
  $\mathbb{C}$ and let ${x^1,...,x^m} \in \mathcal{A}$. Given a
  symmetric $m \times m$ matrix $g=(g_{ij})$ with entries
  $g_{ij}\in\mathcal{A}$, for $i,j=1,...,m$, we say that the triple
  $\mathcal{K}=(\mathcal{A},g,\{x^1,...,x^m\})$ is a K\"ahler--Poisson
  algebra if there exists $\eta\in\mathcal{A}$ such
  that
  \begin{equation} \label{eq:2.1}
    \sum\limits_{i,j,k,l=1}^m
    \eta\{a,x^i\}g_{ij}\{x^j,x^k\}g_{kl}\{x^l,b\}=
    -\{a,b\} \end{equation} for all $a,b$ $\in$ $\mathcal{A}$. We call
  equation \eqref{eq:2.1} \emph{"the K\"ahler--Poisson condition" }.
\end{definition}

Given a K\"ahler-Poisson algebra
$\mathcal{K}=(\mathcal{A},g,\{x^1,...,x^m\})$, let $\mathfrak{g}$
denote the $\mathcal{A}$-module generated by all
$\emph{inner derivations}$ , i.e.
\begin{center}  
  $\mathfrak{g}=\{a_1\{c^1,\cdot\}+...+a_N\{c^N,\cdot\}:$
  $a_i,c^i \in \mathcal{A}$ and $N\in \mathbb{N}\}$.
\end{center}
It is a standard fact that $\mathfrak{g}$ is a Lie algebra with
respect to the bracket
\begin{center}
  $[\alpha,\beta](a)=\alpha(\beta(a))-\beta(\alpha(a))$,
\end{center}
where $\alpha,\beta \in \mathfrak{g}$ and $a \in \mathcal{A}$ (see
e.g. \cite{helgason}).

Moreover, it was shown in \cite{algebras} that $\mathfrak{g}$ is a
projective module and that every K\"ahler--Poisson algebra is a
Lie-Rinehart algebra. For more details on Lie-Rinehart algebras, we
refer to \cite{h:extensions.rinehart,rinehart}. It was also proven
that the matrix $g$ defines a metric (in the context of metric
Lie-Rinehart algebras \cite{algebras}) on $\mathfrak{g}$
via \begin{center}
  $g(\alpha,\beta)=\sum\limits_{i,j=1}^{m}\alpha(x^i)g_{ij}\beta(x^j)$,
\end{center}
for $\alpha,\beta \in \mathfrak{g}$.  Denoting by $\mathcal{P}$
the matrix with entries $\mathcal{P}^{ij}=\{x^i,x^j\}$, the
K\"ahler-Poisson condition \eqref{eq:2.1} %in Definition \ref{def1},
can be written in matrix notation as
\begin{center}
  $\eta\mathcal{P}g\mathcal{P}g\mathcal{P}=-\mathcal{P}$,
\end{center}
if the algebra $\mathcal{A}$ is generated by $\{x^1,...,x^m\}$. Let us
now recall the concept of a morphism of K\"ahler-Poisson algebras
\cite{thesis,morphism}, starting from the following definition.

\begin{definition}
  For a K\"ahler-Poisson algebra
  $\mathcal{K}=(\mathcal{A},g,\{x^1,...,x^m\})$, let
  $\mathcal{A}_{\text{fin}} \subseteq \mathcal{A}$ denote the subalgebra
  generated by $\{x^1,...,x^m\}$.
\end{definition}

\noindent
Clearly, if $\mathcal{A}$ is generated by $\{x^1,\ldots,x^m\}$, which
is often the case in particular examples, then
$\mathcal{A}_{\text{fin}}=\mathcal{A}$. Note that
$\mathcal{A}_{\text{fin}}$ is not necessarily a Poisson subalgebra of
$\mathcal{A}$ in the general case.

\begin{definition}\label{definition}
  Let $\mathcal{K}=(\mathcal{A},g,\{x^1,...,x^m\})$ and
  $\mathcal{K}^{\prime}=(\mathcal{A}^{\prime},g^{\prime},\{y^1,...,{y^m}^{\prime}\})$
  be K\"ahler-Poisson algebras together with their modules of inner
  derivations $\mathfrak{g}$ and $\mathfrak{g}^{\prime}$,
  respectively.  A morphism of K\"ahler-Poisson algebras is a pair of
  maps $(\phi,\psi)$, with
  $\phi:\mathcal{A}\rightarrow \mathcal{A}^{\prime}$ a Poisson algebra
  homomorphism and
  $\psi:\mathfrak{g}\rightarrow \mathfrak{g}^{\prime}$ a Lie algebra
  homomorphism, such that
  
  \begin{enumerate}
  \item \label{1}$\psi(a\alpha)=\phi(a)\psi(\alpha)$,
  \item \label{2}$\phi(\alpha(a))=\psi(\alpha)(\phi(a))$,
  \item \label{3}$\phi(g(\alpha,\beta))=g^{\prime}(\psi(\alpha),\psi(\beta))$,
  \item \label{4}$\phi(\mathcal{A}_\text{fin})\subseteq \mathcal{A}^{\prime}_{\text{fin}} $,
  \end{enumerate}
  for all $a \in \mathcal{A}$ and $\alpha,\beta \in \mathfrak{g}$.  
\end{definition}

\noindent
Let us also recall the following result from \cite{thesis,morphism}, where a
condition for two K\"ahler-Poisson algebras to be isomorphic is
formulated. In this paper we shall repeatedly make use of this result to
understand when two K\"ahler-Poisson algebras are isomorphic for
different choices of metrics.

\begin{proposition}\label{proposition2.2}%\cite{thesis}
  Let $\mathcal{K}=(\mathcal{A},g,\{x^1,...,x^m\})$ and
  $\mathcal{K}^{\prime}=(\mathcal{A}^{\prime},g^{\prime},\{y^1,...,{y^m}^{\prime}\})$
  be K\"ahler-Poisson algebras. Then $\mathcal{K}$ and
  $\mathcal{K}^{\prime}$ are isomorphic if and only if there exists a
  Poisson algebra isomorphism
  $\phi:\mathcal{A}\rightarrow \mathcal{A}^{\prime}$ such that
  $\phi(\mathcal{A}_\textnormal{fin})= \mathcal{A}^{\prime}_\textnormal{fin} $,
  and
  \begin{align}
    \label{properties of isomorphisms}
    \mathcal{P}^{\prime}g^{\prime}\mathcal{P}^{\prime}&=\mathcal{P}^{\prime}A^T\phi(g)A\mathcal{P}^{\prime},
  \end{align} 
  where ${A^i}_\alpha=\frac{\partial \phi(x^i)}{\partial y^\alpha}$ and  $(\mathcal{P}^{\prime})^{\alpha\beta}=\{y^\alpha,y^\beta\}^{\prime}$.
\end{proposition}

\noindent
In what follows, the matrix $\mathcal{P}^\prime$ will be invertible,
implying that \eqref{properties of isomorphisms} is equivalent to
$g^\prime=A^T\phi(g)A$. We shall also need the following result
\cite{thesis,morphism}.

\begin{proposition}
	\label{proposition2.5}
	Let $\mathcal{K}=(\mathcal{A},g,\{x^1,...,x^m\})$ and $\mathcal{K}^{\prime}=(\mathcal{A}^{\prime},g^{\prime},\{y^1,...,{y^m}^{\prime}\})$ be  K\"ahler-Poisson algebras and let $(\phi,\psi):\mathcal{K}\rightarrow \mathcal{K}^{\prime}$ be an isomorphism of  K\"ahler-Poisson algebras.
	If 
	\begin{center}
		$\eta\mathcal{P}g\mathcal{P}g\mathcal{P}=-\mathcal{P}$ and $\medskip$  $\eta^{\prime}\mathcal{P}^{\prime}g^{\prime}\mathcal{P}^{\prime}g^{\prime}\mathcal{P}^{\prime}=-\mathcal{P}^{\prime}$ 
	\end{center}
	then $(\phi(\eta)-\eta^{\prime})\mathcal{P}^\prime=0$.
\end{proposition}

\noindent
Note that, in the current situation, Proposition \ref{proposition2.5}
implies that $\phi(\eta)=\eta^{\prime}$ since $\mathcal{P}^{\prime}$
is invertible.
\section{K\"ahler-poisson algebras over rational functions}

\noindent
Let us start by considering $\complex[x,y]$ together with a (non-zero)
linear Poisson structure; i.e a Poisson bracket determined by
\begin{align*}
  \{x,y\} = \lambda x + \mu y
\end{align*}
for $\lambda,\mu\in\complex$ such that at least one of $\lambda,\mu$
is non-zero (note that the Jacobi identity is satisfied for all
choices of $\lambda$ and $\mu$). The corresponding Poisson algebras
are isomorphic for different choices of $\lambda$ and $\mu$, and for
definiteness we shall choose a particular presentation.
\begin{proposition}
  Let $\mathcal{A}_1 = (\complex[x,y],\{\cdot,\cdot\}_1)$ denote the
  Poisson algebra defined by $\{x,y\}_1=\lambda x+\mu y$ for
  $\lambda,\mu\in\complex$ such that $\{x,y\}_1\neq 0$. Then
  $\mathcal{A}_1$ is isomorphic to the Poisson algebra
  $\mathcal{A}=(\complex[x,y],\{\cdot,\cdot\})$ with $\{x,y\}=x$.
\end{proposition}

\begin{proof}
  We will show that for every choice of $\lambda,\mu\in\complex$ (with
  at least one of them being non-zero), there exists a Poisson algebra
  automorphism $\phi:\mathcal{A}_1\to\mathcal{A}$ of the form
  \begin{align*}
    \phi(x)=ax+by\qquad\phi(y)=cx+dy
  \end{align*}
  with $a,b,c,d\in\complex$ and $ad-bc\neq 0$. Thus, one needs to
  prove that
  \begin{align}
    \{\phi(x),\phi(y)\}=\phi(\{x,y\}_1)\label{eq:A.A1.auto}
  \end{align}
  for every allowed choice of $\lambda,\mu\in \mathbb{C}$. Starting from the
  left hand side we get
  \begin{align*}
    \{\phi(x),\phi(y)\}&=\{ax+by,cx+dy\}=ad\{x,y\}+bc\{y,x\}\\
                       &=ad\{x,y\}-bc\{x,y\}=(ad-bc)x.
  \end{align*}
  From the right hand side, we get 
  \begin{align*}
    \phi(\{x,y\}_1)&=\phi(\lambda x+\mu y)=\lambda\phi(x)+\mu\phi(y)=\lambda(ax+by)+\mu(cx+dy)\\
                   &=\lambda ax+\lambda by+\mu cx+\mu dy
  \end{align*}
  That is, \eqref{eq:A.A1.auto} is equivalent to
  \begin{align}
    &\lambda a+\mu c=ad-bc\label{eq3.1}\\
    &\lambda b+\mu d=0\label{eq3.2}
  \end{align}
  If $\lambda=0$ then $\mu\neq0$, since $\{x,y\}_1\neq0$. From
  \eqref{eq3.2} one gets that $d=0$ and inserting this into \eqref{eq3.1} one
  obtains $\mu c=bc$ implying that $b=-\mu$ (note that $c\neq0$, since
  $ad-bc\neq0$). Hence, choosing e.g. $c=1$ and $a=0$, gives
  $\phi(x)=-\mu y$ and $\phi(y)=x$ defining a Poisson algebra
  isomorphism.

  If $\lambda\neq0$ then from \eqref{eq3.2} one gets
  $b=\frac{-\mu d}{\lambda}$ and inserting this into \eqref{eq3.1} one
  obtains
  \begin{align*}
    \lambda a+\mu c &= ad+\frac{\mu cd}{\lambda}\equivalent
                      a(\lambda-d) = \frac{\mu c}{\lambda}(d-\lambda),
    %a(\lambda-d) &= \mu c\left(\frac{d}{\lambda}-1\right) \\    
  \end{align*}
  and choosing $d=\lambda$ we get $b=-\mu$ (note that $a\neq0$ since
  $ad-bc\neq0$). Hence, choosing e.g. $a=1$ and $c=0$ gives
  $\phi(x)=x-\mu y$ and $\phi(y)=\lambda y$ defining a Poisson algebra
  isomorphism. Thus, we have shown that for every choice of
  $\lambda,\mu\in\complex$ such that $\lambda x+\mu y\neq 0$, one
  can construct a Poisson algebra isomorphism
  $\phi:\mathcal{A}_1\to\mathcal{A}$.
\end{proof}

\noindent
Now, let $\mathbb{C}(x,y)$ denote the rational functions in $x,y$ and let
$\mathbb{C}(x)$ denote the rational functions in $x$. Any
Poisson structure on $\complex[x,y]$ extends to a Poisson structure on
$\complex(x,y)$ via Leibniz rule
\begin{align*}
  \{p,q^{-1}\} = -\{p,q\}q^{-2}
\end{align*}
for $p,q\in\complex[x,y]$. Thus, in the following, we let
$\mathcal{A}(x,y)$ denote the Poisson algebra
$(\complex(x,y),\{\cdot,\cdot\})$ with $\{x,y\}=x$. Given the Poisson
algebra $\mathcal{A}(x,y)$ we set out to study the possible K\"ahler-Poisson
algebra structures arising from $\mathcal{A}(x,y)$; that is, finding $g_{ij}$
such that $(\mathcal{A}(x,y),g,\{x,y\})$ is a K\"ahler-Poisson algebra.
%satisfies equation \eqref{eq:2.1}.

It is easy to check that for an arbitrary symmetric matrix $g$ one
obtains
\begin{align*}
  \mathcal{P}g\mathcal{P}g\mathcal{P}=-\{x,y\}^2\det(g)\mathcal{P}=-x^2\det(g)\mathcal{P}
\end{align*}
giving $\eta=(x^2\det(g))^{-1}$, implying that
$(\mathcal{A}(x,y),g,\{x,y\})$ is a K\"ahler-Poisson algebra as long
as $\det(g)\neq 0$. Hence, any non-degenerate $(2\times 2)$-matrix
$g$, with entries in $\mathcal{A}(x,y)$, gives rise to a
K\"ahler-Poisson algebra over $\mathcal{A}(x,y)$.

Next, let us recall that all automorphisms of $\mathbb{C}[x,y]$ (see
\cite{uber ganze,automorphism}) are given by compositions of
\begin{center}
  $ \phi(x)=\alpha_1x+\beta_1y+\gamma_1$ and $\phi(y)=\alpha_2x+\beta_2y+\gamma_2$ 
\end{center}
for
$\alpha_1,\beta_1,\gamma_1,\alpha_2,\beta_2,\gamma_2 \in \mathbb{C}$,
with $\alpha_1\beta_2\neq\alpha_2\beta_1$ and
\begin{center}
  $\phi(x)=x$ and $\phi(y)=y+p(x)$ 
\end{center}
for all $p(x)\in \mathbb{C}[x]$. In order to use these to construct
K\"ahler-Poisson algebra morphisms, we need to check which ones that
are Poisson algebra morphisms.

\begin{lemma}\label{lemma1}
  Let $\mathcal{A}(x,y)=\mathbb{C}(x,y)$ be the rational functions in $x,y$ with a Poisson structure given by $\{x,y\}=x$. Then:
  \begin{enumerate}[(A)]
  \item $ \phi(x)=\alpha_1x+\beta_1y+\gamma_1$ and
    $\phi(y)=\alpha_2x+\beta_2y+\gamma_2$, for
    $\alpha_1,\beta_1,\gamma_1,\alpha_2,\beta_2,\gamma_2 \in \mathbb{C}$
    with $\alpha_1\beta_2\neq\alpha_2\beta_1$, is a Poisson algebra
    automorphism of $\mathcal{A}(x,y)$ if $\beta_1=\gamma_1=0$ and
    $\beta_2=1$, giving $\phi(x)=\alpha_1x$ and
    $\phi(y)=\alpha_2x+y+\gamma_2$.
  \item $\phi(x)=\alpha x$ and $\phi(y)=y+p(x)$ is a Poisson algebra
    automorphism for all $p(x)\in \mathbb{C}[x]$ and
    $\alpha\in \mathbb{C}\backslash\{0\}$.
  \end{enumerate}
\end{lemma}

\begin{proof}
(A) For $\phi$ to be a Poisson algebra automorphism we need to check that $\phi(\{x,y\})=\{\phi(x),\phi(y)\}$ and since $\{x,y\}=x$ this is equivalent to  $\phi(x)=\{\phi(x),\phi(y)\}$. We start from 
\begin{align*}
\{\phi(x),\phi(y)\}&=\{\alpha_1x+\beta_1y+\gamma_1,\alpha_2x+\beta_2y+\gamma_2\}\\&=\alpha_1\alpha_2\{x,x\}+\alpha_1\beta_2\{x,y\}+\beta_1\beta_2\{y,y\}+\beta_1\alpha_2\{y,x\}\\&=\alpha_1\beta_2\{x,y\}+\beta_1\alpha_2\{y,x\}=(\alpha_1\beta_2-\beta_1\alpha_2)x.
\end{align*}
Now, $\phi(x)=\{\phi(x),\phi(y)\}$ gives
$\alpha_1x+\beta_1y+\gamma_1=(\alpha_1\beta_2-\beta_1\alpha_2)x$.  and
one obtains
\begin{align*}
 \alpha_1=\alpha_1\beta_2-\beta_1\alpha_2,\quad \beta_1=0,\quad\gamma_1=0,  
\end{align*}
implying that $\beta_2=1$ since $\alpha_1\neq0$ (by the assumption
$\alpha_1\beta_2-\alpha_2\beta_1\neq0$). Hence, we get
$ \phi(x)=\alpha_1x$ and $\phi(y)=\alpha_2x+y+\gamma_2$.

(B) For $\phi$ to be a Poisson algebra automorphism we need to check
that $\phi(\{x,y\})=\{\phi(x),\phi(y)\}$ and since $\{x,y\}=x$ we show
that $\phi(x)=\{\phi(x),\phi(y)\}$. We start from the right side
\begin{align*}
  \{\phi(x),\phi(y)\}=\{\alpha x,y+p(x)\}=\alpha\{x,y\}+\alpha\{x,p(x)\}=\alpha x=\phi(x)
\end{align*}
since $\{x,p(x)\}=0$ for arbitrary $p(x)$.  Hence, $\phi$ is a
Poisson algebra automorphism for an arbitrary $p(x)\in\mathbb{C}[x]$
and $\alpha\in \mathbb{C}\backslash\{0\}$.
\end{proof}

\noindent
Next, let us show that compositions of Poisson algebra
automorphisms in Lemma~\ref{lemma1} may be written in a simple form.

\begin{proposition}
  Let $\phi=\phi_1\circ\phi_2\circ...\circ\phi_n$ be an arbitrary
  composition of automorphisms, where each $\phi_k$ can be written as:
  either $\phi(x)=\alpha_1 x$, $\phi(y)=\alpha_2 x+y+\gamma_2$ or
  $\phi(x)=\alpha x$, $\phi(y)=y+p(x)$. Then there exists
  $\alpha\in \mathbb{C}$ and $p(x)\in\mathbb{C}[x]$ such that
  $\phi(x)=\alpha x$ and $\phi(y)=y+p(x)$.
\end{proposition}

\begin{proof}
  We show that every composition of all Poisson algebra automorphisms
  in Lemma \ref{lemma1} can be written with the form
  $\phi(x)=\alpha x$ and $\phi(y)=y+p(x)$. Clearly, both type (A) and
  (B) Poisson algebra automorphisms in Lemma~\ref{lemma1} can be
  written in this form. Thus, let $\phi_1(x)=\alpha_1 x$,
  $\phi_1(y)=y+p_1(x)$, $\phi_2(x)=\alpha_2 x$ and
  $\phi_2(y)=y+p_2(x)$. Then
  \begin{enumerate}
  \item $\phi_1(\phi_2(x))=\phi_1(\alpha_2 x)=\alpha_2\phi_1(x)=\alpha_2\alpha_1 x=\alpha x$.
  \item $\phi_1(\phi_2(y))=\phi_1(y+p_2(x))=y+p_1(x)+p_2(\alpha_1x)=y+p(x)$. 
  \end{enumerate} 
  which are again of the form $\phi(x)=\alpha x$ and $\phi(y)=y+p(x)$,
  implying that an arbitrary composition of automorphisms is of the
  same form.
\end{proof}

In order to prove results related to arbitrary automorphisms of
$\mathcal{A}(x,y)$, we will need to consider the case when
$\phi(x)\in\complex(x)$. In this case, the possible types of
automorphisms can be explicitly described.

\begin{proposition}
  \label{proposition3.4}
  Let $\phi:\mathcal{A}(x,y)\rightarrow\mathcal{A}(x,y)$ be a Poisson
  algebra automorphism such that $\phi(x)\in\mathbb{C}(x)$. Then there
  exist $\alpha,\beta,\gamma,\delta\in \mathbb{C}$ and
  $r(x)\in \mathbb{C}(x)$ such that $\alpha\delta-\beta\gamma\neq0$ and
  \begin{align*}
    &\phi(x)=\frac{\alpha x+\beta}{\gamma x+\delta}\\
    &\phi(y)=\frac{(\alpha x+\beta)(\gamma x+\delta)y}{(\alpha
    \delta-\beta\gamma)x}+r(x).
  \end{align*}
\end{proposition}

\begin{proof}
  If $\phi$ is invertible then, since $\phi(x)\in\mathbb{C}(x)$,
  $\phi(x)$ has to be an invertible rational function of $x$. It
  is well known that such a function is of the form
  $\phi(x)=\frac{\alpha x+\beta}{\gamma x+\delta}$ for some
  $\alpha,\beta,\gamma,\delta\in \mathbb{C}$ with
  $\alpha\delta-\beta\gamma\neq0$. Since $\phi$ is a Poisson algebra
  automorphism, one can determine the possible $\phi(y)$ via
  \begin{align*}
    \{\phi(x),\phi(y)\}&=\phi(\{x,y\}).
  \end{align*}
  We start from the left hand side
  \begin{align*}
    \{\phi(x),\phi(y)\}&=\Big\{\frac{\alpha x+\beta}{\gamma x+\delta},\phi(y)\Big\}=
                         \Big\{\frac{\alpha x+\beta}{\gamma x+\delta},y\Big\}\phi(y)^\prime_ y=
                         \Big(\frac{\alpha x+\beta}{\gamma x+\delta}\Big)^\prime_x\{x,y\}\phi(y)^\prime_y\\
                       &=\Big(\frac{\alpha\delta-\beta\gamma}{(\gamma x+\delta)^2}\Big)x\phi(y)^\prime_y,
  \end{align*}
  and from the right hand side we get 
  \begin{align*}
    \phi(\{x,y\})=\phi(x)=\frac{\alpha x+\beta}{\gamma x+\delta}.
  \end{align*}
  Therefore, we obtain
  \begin{align*}
    \frac{\alpha\delta-\beta\gamma}{\gamma x+\delta}x\phi(y)^\prime_y&=\alpha x+\beta\implies\\
    (\alpha\delta-\beta\gamma)x\phi(y)^\prime_y&=(\alpha x+\beta)(\gamma x+\delta)\implies\\
    \phi(y)&=\frac{(\alpha x+\beta)(\gamma x+\delta)y}{(\alpha \delta-\beta \gamma)x}+r(x).
  \end{align*}
  for some $r(x)\in\complex(x)$.
\end{proof}

Let us now start to investigate isomorphism classes of metrics for
K\"ahler-Poisson algebras over $\mathcal{A}(x,y)$. The simplest case
is when the metrics are constant; i.e 
\begin{align*}
  g=\begin{pmatrix}
    a & b  \\
    b & c \\
  \end{pmatrix}
  \qquad\text{and}\qquad \tilde{g}=
  \begin{pmatrix}
    \tilde{a} & \tilde{b}  \\
    \tilde{b} & \tilde{c} \\ 
  \end{pmatrix}
\end{align*}
for $a,b,c,\tilde{a},\tilde{b},\tilde{c}\in \mathbb{C}$.

\begin{proposition}\label{prop:constant.metric}
  Let $\mathcal{K}=(\mathcal{A}(x,y),g,\{x,y\})$ and
  $\tilde{\mathcal{K}}=(\mathcal{A}(x,y),\tilde{g},\{x,y\})$ be
  K\"ahler-Poisson algebras, with
  \begin{center}
    $g=\begin{pmatrix}
      a & b  \\
      b & c \\ 
    \end{pmatrix}$ and $\tilde{g}=\begin{pmatrix}
      \tilde{a} & \tilde{b}  \\
      \tilde{b} & \tilde{c} \\ 
    \end{pmatrix}$
  \end{center}
  where $a,b,c,\tilde{a},\tilde{b},\tilde{c}\in \mathbb{C}$ such that
  $\text{det}(g)\neq0$ and $\text{det}(\tilde{g})\neq0$. Then
  $\mathcal{K}\cong \tilde{\mathcal{K}}$ if and only if $c=\tilde{c}$.
\end{proposition}

\begin{proof}
  First, we assume that $c=\tilde{c}$ and show that
  $\mathcal{K}\cong \tilde{\mathcal{K}}$ by using an automorphism of
  the form $\phi(x)=\alpha x$ and $\phi(y)=y+p(x)$
  (cf. Lemma~\ref{lemma1}). We will do this by applying
  Proposition~\ref{proposition2.2} to show that there exists
  $\alpha\in\complex$ and $p(x)\in\complex[x]$ such that
  \begin{align}\label{eq2}
    \tilde{g}=A^T\phi(g)A.
  \end{align}
  First, note that $\phi(g)=g$, since $\phi$ is unital. From
  $A^i_{\alpha}=\frac{\partial\phi(x^i)}{\partial y^{\alpha}}$ one
  computes
  \begin{center}
    $A=\begin{pmatrix}
      \frac{\partial\phi(x)}{\partial x} & \frac{\partial\phi(x)}{\partial y}  \\
      \frac{\partial\phi(y)}{\partial x} & \frac{\partial\phi(y)}{\partial y} \\ 
    \end{pmatrix}=\begin{pmatrix}
      \alpha & 0  \\
      p^{\prime}(x) & 1 \\ 
    \end{pmatrix}$,
  \end{center}
  giving \eqref{eq2} as
  \begin{align*}
    \begin{pmatrix}
      \tilde{a} & \tilde{b}  \\
      \tilde{b} & c \\ 
    \end{pmatrix}&= \begin{pmatrix}
      \alpha & p^{\prime}(x)  \\
      0 & 1 \\ 
    \end{pmatrix}\begin{pmatrix}
      a & b  \\
          b & c \\ 
        \end{pmatrix}\begin{pmatrix}
          \alpha & 0  \\
          p^{\prime}(x) & 1 \\ 
        \end{pmatrix}= \begin{pmatrix}
          \alpha a+p^{\prime}(x) b & \alpha b+p^{\prime}(x) c  \\
          b & c \\ 
        \end{pmatrix}\begin{pmatrix}
          \alpha & 0  \\
          p^{\prime}(x) & 1 \\ 
        \end{pmatrix}\\&=\begin{pmatrix}
          \alpha(\alpha a+p^{\prime}(x) b)+p^{\prime}(x)(\alpha b+p^{\prime}(x) c ) & \alpha b+p^{\prime}(x) c  \\
          \alpha b+\mathcal{P}^{\prime}(x) c & c \\ 
        \end{pmatrix},
  \end{align*} 
  showing that \eqref{eq2} is equivalent to
  \begin{align}\label{eq3}
    \tilde{a}=\alpha^2a+2\alpha p^{\prime}(x) b+(p^{\prime}(x))^2c
  \end{align} 
  \begin{align}\label{eq4}
    \tilde{b}=\alpha b+p^{\prime}(x) c
  \end{align}  
  First, assume that $c\neq0$. From \eqref{eq4} we get
  $p^{\prime}(x)=\frac{\tilde{b}-\alpha b}{c}$ giving
  $p(x)=(\frac{\tilde{b}-\alpha b}{c})x$, and inserting
  $p^{\prime}(x)$ in \eqref{eq3} we obtain
  \begin{align*}
    \tilde{a}&=\alpha^2a+2\alpha\Big(\frac{\tilde{b}-\alpha b}{c}\Big) b+\Big(\frac{\tilde{b}-\alpha b}{c}\Big)^2c =\alpha^2a+\frac{2\alpha b \tilde{b}-2\alpha^2 b^2}{c} +\frac{\tilde{b}^2c-2\alpha b \tilde{b}c+\alpha^2 b^2 c}{c^2}
  \end{align*} 
  Multiplying both sides by $c^2$, we get
  \begin{align*}
    &\tilde{a}c^2=\alpha^2ac^2+2\alpha b \tilde{b}c-2\alpha^2 b^2c+\tilde{b}^2c-2\alpha b \tilde{b}c+\alpha^2 b^2 c\implies
    \alpha^2=\frac{\tilde{a}c-\tilde{b}^2}{ac-b^2},
  \end{align*}
  where $ac-b^2=\text{det}(g)\neq0$ by assumption. Hence, for
  $c=\tilde{c}\neq0$, we have constructed an isomorphism between
  $\mathcal{K}$ and $\tilde{\mathcal{K}}$.  If $c=\tilde{c}=0$ we get
  from \eqref{eq4} that $\alpha=\frac{\tilde{b}}{b}$, where $b\neq0$ since
  $\text{det}(g)\neq0$. Now, we find $p^{\prime}(x)$ from \eqref{eq3}
  by using $\alpha=\frac{\tilde{b}}{b}$
  \begin{align*}
    \tilde{a}&=a\Big(\frac{\tilde{b}}{b}\Big)^2 +2\Big(\frac{\tilde{b}}{b}\Big)p'(x)b+\big(p'(x)\big)^2 c\implies
               p'(x)=\frac{\tilde{a}b^2-a\tilde{b}^2}{2b^2\tilde{b}}            
  \end{align*}
  (note that $\tilde{b}\neq0$ since $\det(\tilde{g})\neq0$), giving
  $p(x)=(\frac{\tilde{a}b^2-a\tilde{b}^2}{2b^2\tilde{b}})x$. Hence,
  for $c=\tilde{c}=0$, this gives an isomorphism between $\mathcal{K}$
  and $\tilde{\mathcal{K}}$.  We conclude that
  $\mathcal{K}\cong \tilde{\mathcal{K}}$ if $c=\tilde{c}$.
  
  Vice versa, assume that $\mathcal{K}\cong \tilde{\mathcal{K}}$. We have
  \begin{center}
    $\eta=\{x,y\}^2\det(g)=x^2\det(g)$ and $\tilde{\eta}=\{x,y\}^2\det(\tilde{g})=x^2\det(\tilde{g})$
  \end{center}
  with $\det(g),\det(\tilde{g})\in\mathbb{C}$. By using Proposition
  \ref{proposition2.5}, stating that $\phi(\eta)=\tilde{\eta}$, one
  obtains
  \begin{align*}
    \phi(x^2)\det(g)&=x^2\det(\tilde{g})\implies
    \frac{\phi(x^2)}{x^2}=\frac{\det\tilde{g}}{\det(g)}\in\complex
  \end{align*}
  implying that $\phi(x)=\alpha x$ for some $\alpha\in
  \mathbb{C}$. Furthermore, using Proposition \ref{proposition3.4}
  with $\beta=0,\gamma=0,\delta=1$ one obtains
  $\phi(y)=y+r(x)$. Hence, any isomorphism have to be of the form
  $\phi(x)=\alpha x$ and $\phi(y)=y+r(x)$, for some
  $\alpha\in\complex$ and $r(x)\in \mathbb{C}(x)$.
     
  Using the above form of the isomorphism in Proposition
  \ref{proposition2.2} one gets
  \begin{align*} 
    \begin{pmatrix}
      \tilde{a} & \tilde{b}  \\
      \tilde{b} & \tilde{c} \\ 
    \end{pmatrix}&= \begin{pmatrix}
      \alpha & r^{\prime}(x)  \\
      0 & 1 \\ 
    \end{pmatrix}\begin{pmatrix}
      a & b  \\
      b & c \\ 
    \end{pmatrix}\begin{pmatrix}
      \alpha & 0  \\
      r^{\prime}(x) & 1 \\ 
    \end{pmatrix}\\&
                =\begin{pmatrix}
                  \alpha(\alpha a+r^{\prime}(x) b)+r^{\prime}(x)(\alpha b+r^{\prime}(x) c ) & \alpha b+r^{\prime}(x) c  \\
                  \alpha b+r^{\prime}(x) c & c \\ 
                \end{pmatrix},
  \end{align*} 
  giving $\tilde{c}=c$.                   
\end{proof}

\noindent
The above result shows that the isomorphism classes of
K\"ahler-Poisson algebras with constant metrics can be parametrized by
one (complex) parameter.  In the next result, we study the case when
the metric only depends on $x$, and start by giving sufficient
conditions for the K\"ahler-Poisson algebras to be isomorphic.

\begin{proposition}\label{proposition x}
  Let $\mathcal{K}=(\mathcal{A}(x,y),g,\{x,y\})$ and
  $\tilde{\mathcal{K}}=(\mathcal{A}(x,y),\tilde{g},\{x,y\})$ be
  K\"ahler-Poisson algebras, with
  \begin{center}
    $g=\begin{pmatrix}
      a(x) & b(x)  \\
      b(x) & c(x) \\ 
    \end{pmatrix}$ and $\tilde{g}=\begin{pmatrix}
      \tilde{a}(x) &\tilde{b}(x)  \\
      \tilde{b}(x) & \tilde{c}(x) \\ 
    \end{pmatrix}$
  \end{center}
  where
  $a(x),b(x),c(x),\tilde{a}(x),\tilde{b}(x),\tilde{c}(x)\in
  \mathbb{C}[x]$ such that $\text{det}(g)\neq0$ and
  $\text{det}(\tilde{g})\neq0$. If $c(x)\neq0$ and there
  exists $\alpha\in\mathbb{C}$ such that:
  \begin{enumerate}
  \item $\big(\tilde{a}(x)-\alpha^2a(\alpha x)\big)c(\alpha x)=\tilde{b}(x)^2-\alpha^2 b(\alpha x)^2$  
    
  \item $\frac{\tilde{b}(x)-\alpha b(\alpha x)}{c(\alpha x)}\in\mathbb{C}[x]$
  \item $\tilde{c}(x)=c(\alpha x)$
  \end{enumerate} 
  then $\mathcal{K}\cong \tilde{\mathcal{K}}$.
  If $\tilde{c}(x)=c(x)=0$ and there exists $\alpha\in\mathbb{C}$ such that: 
  \begin{enumerate}[(a)]
  \item $\tilde{b}(x)=\alpha b(\alpha x)$
  \item $\frac{\tilde{a}(x)-\alpha^2 a(\alpha x)}{2\alpha b(\alpha x)}\in\mathbb{C}[x]$
  \end{enumerate} 
  then $\mathcal{K}\cong \tilde{\mathcal{K}}$.
\end{proposition}

\begin{proof}
  Let $\phi$ be an automorphism of $\phi(x)=\alpha x$ and
  $\phi(y)=y+p(x)$. We will show that one may find $\alpha\in\complex$
  and $p(x)\in\complex[x]$ such that
  \begin{align}\label{eq5.1}
    \tilde{g}=A^T\phi(g)A,
  \end{align}
  implying, via Proposition~\ref{proposition2.2}, that
  $\mathcal{K}\cong \tilde{\mathcal{K}}$.  From
  $A^i_{\alpha}=\frac{\partial\phi(x^i)}{\partial y^{\alpha}}$ one
  computes
\begin{center}
 $A=\begin{pmatrix}
     \frac{\partial\phi(x)}{\partial x} & \frac{\partial\phi(x)}{\partial y}  \\
     \frac{\partial\phi(y)}{\partial x} & \frac{\partial\phi(y)}{\partial y} \\ 
 \end{pmatrix}=\begin{pmatrix}
      \alpha & 0  \\
      p^{\prime}(x) & 1 \\ 
  \end{pmatrix}$,\end{center} giving \eqref{eq5.1} as
  \begin{align*}
    \begin{pmatrix}
      \tilde{a}(x) &\tilde{b}(x)  \\
      \tilde{b}(x) & \tilde{c}(x) \\ 
    \end{pmatrix}&= \begin{pmatrix}
      \alpha & p^{\prime}(x)  \\
      0 & 1 \\ 
    \end{pmatrix}\begin{pmatrix}
      a(\alpha x) & b(\alpha x)  \\
      b(\alpha x) & c(\alpha x) \\ 
    \end{pmatrix}\begin{pmatrix}
      \alpha & 0  \\
      p^{\prime}(x) & 1 \\ 
    \end{pmatrix}\\&
                    =\begin{pmatrix}
                       \alpha^2 a(\alpha x)+2\alpha p^{\prime}(x) b(\alpha x)+p'(x)^2c(\alpha x) & \alpha b(\alpha x)+p^{\prime}(x) c(\alpha x)  \\
                       \alpha b(\alpha x)+p^{\prime}(x) c(\alpha x) & c(\alpha x) \\ 
                     \end{pmatrix}.
  \end{align*} 
  Hence, \eqref{eq5.1} is equivalent to
  \begin{align}
    &\tilde{a}(x)=\alpha^2a(\alpha x)+2\alpha p^{\prime}(x) b(\alpha x)+p^{\prime}(x)^2c(\alpha x)\label{eq6}\\
    &\tilde{b}(x)=\alpha b(\alpha x)+p^{\prime}(x) c(\alpha x)\label{eq7}\\
    &\tilde{c}(x) = c(\alpha x)\label{eq:tildec}
  \end{align}  
  First, assume that $c(x)\neq0$ together with the assumptions (1)--(3).
  From \eqref{eq7} we get
  \begin{align*}
    p'(x)=\frac{\tilde{b}(x)-\alpha b(\alpha x)}{c(\alpha x)}.
  \end{align*}
  and, by assumption, this is in $\mathbb{C}[x]$ which implies that
  one may integrate to get $p(x)$. Inserting $p'(x)$ in
  \eqref{eq6} we obtain
  \begin{align*}
    \tilde{a}(x)c(\alpha x)
    &=\alpha^2a(\alpha x)c(\alpha x)+2\alpha b(\alpha x)p^\prime(x)c(\alpha x)+p^\prime(x)^2c(\alpha x)^2\\
    &=\alpha^2a(\alpha x)c(\alpha x)+2\alpha b(\alpha x)\big(\tilde{b}(x)-\alpha b(\alpha x)\big)+\big(\tilde{b}(x)-\alpha b(\alpha x)\big)^2\\
    % &=\alpha^2a(\alpha x)c(\alpha x)+2\alpha b(\alpha x)\tilde{b}(x)-2\alpha^2 b(\alpha x)^2+\tilde{b}(x)^2+\alpha^2 b(\alpha x)^2-2\alpha\tilde{b}(x)b(\alpha x)\\
    &=\alpha^2a(\alpha x)c(\alpha x)-\alpha^2 b(\alpha x)^2+\tilde{b}(x)^2
  \end{align*}
  that is,
  $\big(\tilde{a}(x)-\alpha^2a(\alpha x)\big)c(\alpha
  x)=\tilde{b}(x)^2-\alpha^2 b(\alpha x)^2$ which is true by
  assumption.

  If $\tilde{c}(x)=c(x)=0$, we assume that
  $\frac{\tilde{a}(x)-\alpha^2 a(\alpha x)}{2\alpha b(\alpha
    x)}\in\mathbb{C}[x]$ and $\tilde{b}(x)=\alpha b(\alpha x)$. Then
  \eqref{eq7} is immediately satisfied and from \eqref{eq6} we get
  $p^{\prime}(x)=\frac{\tilde{a}(x)-\alpha^2 a(\alpha x)}{2\alpha
    b(\alpha x)}$. By assumption this is in $ \mathbb{C}[x]$ which
  implies that one may integrate to get $p(x)$.  This shows that one
  may explicitly construct an isomorphism between $\mathcal{K}$ and
  $\tilde{\mathcal{K}}$, given the assumptions in the statement.
\end{proof}

\noindent
For the sake of illustration, let us use the above result to give a
simple example, and construct two seemingly different metrics that
give rise to isomorphic K\"ahler-Poisson algebras.

\begin{example}
  Let $\mathcal{K}=(\mathcal{A}(x,y),g,\{x,y\})$ and
  $\tilde{\mathcal{K}}=(\mathcal{A}(x,y),\tilde{g},\{x,y\})$ be
  K\"ahler-Poisson algebras, with
  \begin{center} $g=\begin{pmatrix}
      a(x) & 0 \\
      0 & c(x) \\
    \end{pmatrix}$ and $\tilde{g}=\begin{pmatrix}
      a(x)+q(x)^2c(x) &q(x)c(x) \\
      q(x)c(x) & c(x)\\ 
    \end{pmatrix}$
  \end{center}
  for $a(x),c(x),q(x)\in\mathbb{C}[x]$. We will use Proposition
  \ref{proposition x} to show that
  $\mathcal{K}\cong \tilde{\mathcal{K}}$. Let us check conditions
  (1)--(3) with $\alpha=1$, $b(x)=0$, $\tilde{c}(x)=c(x)$, $\tilde{b}(x)=q(x)c(x)$ and
  \begin{align*}
    \tilde{a}(x) = a(x)+q(x)^2c(x).
  \end{align*}
  \begin{enumerate}
  \item $\big(\tilde{a}(x)-\alpha^2a(\alpha x)\big)c(\alpha x)=\tilde{b}(x)^2-\alpha^2 b(\alpha x)^2$:
    \begin{align*}
      &\big(\tilde{a}(x)-\alpha^2a(\alpha x)\big)c(\alpha x) = q(x)^2c(x)^2\\
      &\tilde{b}(x)^2-\alpha^2 b(\alpha x)^2 = q(x)^2c(x)^2,
    \end{align*}
    
  \item $\displaystyle\frac{\tilde{b}(x)-\alpha b(\alpha x)}{c(\alpha x)}=q(x)\in\mathbb{C}[x]$,
    
  \item $\tilde{c}(x)=c(\alpha x)$ since $\alpha=1$ and $\tilde{c}(x)=c(x)$.
  \end{enumerate}
  For instance, with $a(x)=x$, $c(x)=x^2$ and $q(x)=x^3$ one concludes that
  \begin{align*}
    g =
    \begin{pmatrix}
      x & 0 \\
      0 & x^2
    \end{pmatrix}\text{ and }
          \begin{pmatrix}
            x+x^8 & x^5\\
            x^5 & x^2
          \end{pmatrix}
  \end{align*}
  define isomorphic K\"ahler-Poisson algebras.
\end{example}

\noindent
The following result shows that, in the more restricted situation
where the metrics are assumed to be diagonal, one describe all
isomorphism classes.

\begin{proposition}\label{prop:diagonal.g}
  Let $\mathcal{K}=(\mathcal{A}(x,y),g,\{x,y\})$ and
  $\tilde{\mathcal{K}}=(\mathcal{A}(x,y),\tilde{g},\{x,y\})$ be
  K\"ahler-Poisson algebras, with
  \begin{center}
    $g=\begin{pmatrix}
      a(x) & 0  \\
      0 & c(x) \\ 
    \end{pmatrix}$ and $\tilde{g}=\begin{pmatrix}
      \tilde{a}(x) &0  \\
      0 & \tilde{c}(x) \\ 
    \end{pmatrix}$
  \end{center}
  where $a(x),c(x),\tilde{a}(x),\tilde{c}(x)\in \mathbb{C}[x]$ such
  that $\text{det}(g)\neq0$ and $\text{det}(\tilde{g})\neq0$. Then
  $\mathcal{K}\cong \tilde{\mathcal{K}}$ if and only if there exists
  $\alpha\in\mathbb{C}$ such that:
  \begin{enumerate}
  \item $\tilde{c}(x)=c(\alpha x)$
  \item $\tilde{a}(x)=\alpha^2a(\alpha x)$ 
  \end{enumerate} 
\end{proposition}

\begin{proof}
  To show that $\mathcal{K}\cong \tilde{\mathcal{K}}$ if (1) and (2)
  are satisfied, we use Proposition \ref{proposition x}. Namley, 
  \begin{align*}
    \big(\tilde{a}(x)-\alpha^2a(\alpha x)\big)c(\alpha x)=\tilde{b}(x)^2-\alpha^2 b(\alpha x)^2 
  \end{align*}
  for $b=0$ becomes
  \begin{align*}
    \big(\tilde{a}(x)-\alpha^2a(\alpha x)\big)c(\alpha x)=0
  \end{align*}
  which is satisfied since $\tilde{a}(x)=\alpha^2a(\alpha
  x)$. (Condition (2) in Proposition~\ref{proposition x} is trivially
  satisfied since $b=\tilde{b}=0$.)
	
  Vice versa, assume that $\mathcal{K}\cong \tilde{\mathcal{K}}$. We
  have
  \begin{center}
    $\eta=\{x,y\}^2\det(g)=x^2\det(g)$ and $\tilde{\eta}=\{x,y\}^2\det(\tilde{g})=x^2\det(\tilde{g})$
  \end{center}
  with $\det(g),\det(\tilde{g})\in\mathbb{C}[x]$. By using Proposition
  \ref{proposition2.5}, which gives that $\phi(\eta)=\tilde{\eta}$,
  it follows that
  \begin{align*}
    \phi(x^2)\det(g)&=x^2\det(\tilde{g})
    %\frac{\phi(x^2)}{x^2}&=\frac{\det\tilde{g}}{\det(g)}
  \end{align*}
  implying that $\phi(x)\in\mathbb{C}(x)$. By
  Proposition~\ref{proposition3.4}, if $\phi(x)\in\mathbb{C}(x)$, then
  the automorphism has to be of the form
  \begin{center}
    $\phi(x)=\frac{\alpha x+\beta}{\gamma x+\delta}$ and $\phi(y)=\frac{(\alpha x+\beta)(\gamma x+\delta)y}{(\alpha\delta-\beta\gamma)x}+r(x)$
  \end{center} 
  where, $\alpha,\beta,\gamma,\delta\in\mathbb{C}$,
  $r(x)\in\mathbb{C}(x)$ and
  $\alpha\delta-\beta\gamma\neq0$. Moreover, by Proposition
  \ref{proposition2.2}, one also has
  \begin{align}\label{eq5}
    \tilde{g}=A^T\phi(g)A.
  \end{align} 
  with
  \begin{align*}
    A=
    \begin{pmatrix}
      \phi(x)^{\prime}_x & \phi(x)^{\prime}_y  \\
      \phi(y)^{\prime}_x & \phi(y)^{\prime}_y \\ 
    \end{pmatrix}
    =
    \begin{pmatrix}
      \phi(x)^{\prime}_x & 0  \\
      \phi(y)^{\prime}_x & \phi(y)^{\prime}_y \\ 
    \end{pmatrix}
  \end{align*}
  since $\phi(x)'_y=0$, giving% \eqref{eq5} as
  \begin{align*}
    \begin{pmatrix}
      \tilde{a}(x) &0  \\
      0 & \tilde{c}(x)  \\ 
    \end{pmatrix}&= \begin{pmatrix}
      \phi(x)^{\prime}_x & \phi(y)^{\prime}_x  \\
      0 & \phi(y)^{\prime}_y \\
    \end{pmatrix}\begin{pmatrix}
      \phi(a(x)) & 0  \\
      0 & \phi(c(x)) \\ 
    \end{pmatrix}\begin{pmatrix}
      \phi(x)^{\prime}_x & 0  \\
      \phi(y)^{\prime}_x & \phi(y)^{\prime}_y \\ 
    \end{pmatrix}\\
    &=
    \begin{pmatrix}
      \phi(a(x))\paraa{\phi(x)'_x}^2+\phi(c(x))\paraa{\phi(y)'_x}^2 &
      \phi(c(x))\phi(y)'_x\phi(y)'_y\\
      \phi(c(x))\phi(y)'_x\phi(y)'_y &
      \phi(c(x))\paraa{\phi(y)'_y}^2
    \end{pmatrix}.
  \end{align*} 
  Hence, \eqref{eq5} is equivalent to
  \begin{align}
    &\tilde{a}(x)=\phi(a(x))(\phi(x)^{\prime}_x)^2+\phi(c(x))(\phi(y)^{\prime}_x)^2 \label{eq:eq6}\\
    &\tilde{c}(x)=\phi(c(x))(\phi(y)^{\prime}_y )^2\label{eq:other.ctilde}\\
    &\phi(c(x))\phi(y)^{\prime}_x\phi(y)^{\prime}_y =0.\label{eq8}
  \end{align}
  Now, $\phi(c(x))\neq 0$ (since $\det(g)\neq 0)$ and
  $\phi(y)'_y\neq 0$ (since $\alpha\delta-\beta\gamma\neq 0$) which
  implies, by \eqref{eq8}, that
  \begin{center}
    $\phi(y)'_x = \Big(\frac{(\alpha x+\beta)(\gamma x+\delta)}{(\alpha\delta-\beta\gamma)x}y\Big)^{\prime}_x+r^{\prime}(x)=0$.
  \end{center}
  It follows that $r(x)=r_0\in\complex$ and 
  \begin{align*}
    \Big(\frac{(\alpha x+\beta)(\gamma x+\delta)}{(\alpha\delta-\beta\gamma)x}y\Big)^{\prime}_x&=0\implies
    \frac{(\alpha x+\beta)(\gamma x+\delta)}{(\alpha\delta-\beta\gamma)x}=\lambda\in\mathbb{C},
  \end{align*}
  yielding
  \begin{align*}
    (\alpha x+\beta)(\gamma x+\delta)&=\lambda x(\alpha\delta-\beta\gamma)
  \end{align*}
  and consequently
  \begin{center}
    $\alpha\gamma=0$, $\beta\delta=0$ and $\alpha\delta+\beta\gamma=\lambda (\alpha\delta-\beta\gamma)$.
  \end{center}
  If $\alpha=0$, then $\beta\neq0,\gamma\neq0$ (since
  $\alpha\delta-\beta\gamma\neq 0$) implying that $\delta=0$ and
  therefore, $\beta\gamma=-\lambda\beta\gamma$ which implies that
  $\lambda=-1$. Hence, the automorphism have to be of the form
  $\phi(x)=\frac{\beta}{\gamma x}$ and $\phi(y)=-y+r_0$. Using
  Proposition \ref{proposition2.2} we get
  \begin{align*}
    \begin{pmatrix}
      \tilde{a}(x) &0  \\
      0 & \tilde{c}(x)  \\ 
    \end{pmatrix}&= \begin{pmatrix}
      -\frac{\beta}{\gamma x} & 0 \\
      0 & -1 \\
    \end{pmatrix}\begin{pmatrix}
      \phi(a(x)) & 0  \\
      0 & \phi(c(x)) \\ 
    \end{pmatrix}\begin{pmatrix}
      -\frac{\beta}{\gamma x} & 0 \\
      0 & -1
    \end{pmatrix}\\&= \begin{pmatrix}
      -\frac{\beta}{\gamma x}\phi(a(x)) & 0 \\
      0& -\phi(c(x))\\
    \end{pmatrix}\begin{pmatrix}
      -\frac{\beta}{\gamma x} & 0 \\
      0 & -1
    \end{pmatrix}\\&=\begin{pmatrix}
      \phi(a(x))\para{-\frac{\beta}{\gamma x}}^2 & 0 \\
      0& \phi(c(x))\end{pmatrix}
  \end{align*}
  giving that
  \begin{align*}
    \tilde{a}(x)&=\phi(a(x))\para{-\frac{\beta}{\gamma x}}^2
                  =a(\phi(x))\para{-\frac{\beta}{\gamma x}}^2
                  =\para{-\frac{\beta}{\gamma x}}^2a\para{\frac{\beta}{\gamma x}}
  \end{align*} 
  However, the above form of the automorphism is not possible since we
  have assumed that $a(x),\tilde{a}(x)$ are non-zero polynomials (and
  not rational functions).
 
  If $\alpha\neq0$, then $\gamma=0,\delta\neq0$ implying that
  $\beta=0$ and $\alpha\delta=\lambda\alpha\delta$, which implies that
  $\lambda=1$.  Hence, the automorphism have to be of the form
  $\phi(x)=\frac{\alpha x}{\delta }=\tilde{\alpha}x$ and
  $\phi(y)=\lambda y=y+r_0$. Using Proposition \ref{proposition2.2} we
  get
  \begin{align*}
    \begin{pmatrix}
      \tilde{a}(x) &0  \\
      0 & \tilde{c}(x)  \\ 
    \end{pmatrix}
    &= \begin{pmatrix}
      \tilde{\alpha} & 0 \\
      0 & -1 \\
    \end{pmatrix}
    \begin{pmatrix}
      \phi(a(x)) & 0  \\
      0 & \phi(c(x)) \\ 
    \end{pmatrix}
    \begin{pmatrix}
      \tilde{\alpha} & 0 \\
      0 & -1
    \end{pmatrix}\\&
                   =\begin{pmatrix}
      \tilde{\alpha}^2\phi(a(x)) & 0 \\
      0& \phi(c(x))\end{pmatrix}
  \end{align*}
  giving that
  \begin{align*}
    \tilde{a}(x)&=\tilde{\alpha}^2\phi(a(x))=\tilde{\alpha}^2a(\phi(x))=\tilde{\alpha}^2a(\tilde{\alpha}x)\\
    \tilde{c}(x)&=\phi(c(x))=c(\phi(x))=c(\tilde{\alpha}x)
  \end{align*}
  which concludes the proof of the statement.
\end{proof}

\noindent Let us give another simple example of isomorphic K\"ahler-Poisson algebras.

\begin{example}
  Let $\mathcal{K}=(\mathcal{A}(x,y),g,\{x,y\})$ and
  $\tilde{\mathcal{K}}=(\mathcal{A}(x,y),\tilde{g},\{x,y\})$ be
  K\"ahler-Poisson algebras, with
  \begin{align*}
    g=\begin{pmatrix}
      a(x) & 0 \\
      0 & c(x) \\
    \end{pmatrix}\text{ and }
    \tilde{g}=\begin{pmatrix}
      \alpha^2a(\alpha x) &0 \\
      0 & c(\alpha x)\\ 
    \end{pmatrix} 
  \end{align*}
  for $a(x),c(x)\in\complex[x]$ and
  $\alpha\in\complex$. Proposition~\ref{prop:diagonal.g} shows that
  $\mathcal{K}\cong \tilde{\mathcal{K}}$.
  For instance, with $a(x)=x$, $c(x)=1+x+x^2$ and $\alpha=-2$ one finds that
  \begin{align*}
    g=\begin{pmatrix}
      x & 0 \\
      0 & 1+x+x^2 \\ 
    \end{pmatrix}\text{ and }
    \tilde{g}=
    \begin{pmatrix}
      -8x &0 \\
      0 & 1-2x+4x^2\\ 
    \end{pmatrix}
  \end{align*}
  give isomorphic K\"ahler-Poisson algebras.
\end{example}

\noindent
For general metrics, the situation becomes much more
complicated. However, let us finish by giving a sufficient condition
for diagonal metrics depending on $y$.
\begin{proposition}\label{prop:diagonal.y}
  Let $\mathcal{K}=(\mathcal{A}(x,y),g,\{x,y\})$ and $\tilde{\mathcal{K}}=(\mathcal{A}(x,y),\tilde{g},\{x,y\})$ be K\"ahler-Poisson algebras, with
  \begin{center}
    $g(y)=\begin{pmatrix}
      a(y) & 0  \\
      0 & c(y) \\ 
    \end{pmatrix}$ and $\tilde{g}(y)=\begin{pmatrix}
      \tilde{a}(y) &0  \\
      0 & \tilde{c}(y) \\ 
    \end{pmatrix}$
  \end{center}
  for $a(y),c(y),\tilde{a}(y),\tilde{c}(y)\in \mathbb{C}[x]$. If there exists $\alpha,\lambda\in\mathbb{C}$ such that
  \begin{enumerate}
  \item $\tilde{a}(y)=\alpha^2a(y+\lambda)$
  \item $\tilde{c}(y)=c(y+\lambda)$ 
  \end{enumerate}
  then $\mathcal{K}\cong \tilde{\mathcal{K}}$.
\end{proposition}

\begin{proof}
  Assume $\tilde{a}(y)=\alpha^2a(y+\lambda)$ and
  $\tilde{c}(y)=c(y+\lambda)$. To prove that
  $\mathcal{K}\cong \tilde{\mathcal{K}}$, we will use
  Proposition~\ref{proposition2.2} and show that
  \begin{align}\label{eq9}
    \tilde{g}=A^T\phi(g)A,
  \end{align}
  for an automorphism of the type $\phi(x)=\alpha x$ and
  $\phi(y)=y+p(x)$.  Let $p(x)=\lambda$, then $p^{\prime}(x)=0$ and we
  compute
  \begin{center}
    $A=\begin{pmatrix}
      \alpha & 0  \\
      p^{\prime}(x) & 1 \\ 
    \end{pmatrix}=\begin{pmatrix}
      \alpha & 0  \\
      0  & 1 \\ 
    \end{pmatrix}$
  \end{center}
  giving \eqref{eq9} as
  \begin{align*}
\begin{pmatrix}
\tilde{a}(y) &0  \\
0 & \tilde{c}(y) \\ 
\end{pmatrix}&=\begin{pmatrix}
\alpha &0  \\
0 & 1 \\ 
\end{pmatrix}\begin{pmatrix}
\phi(a(y)) & 0  \\
0 & \phi(c(y)) \\ 
\end{pmatrix}\begin{pmatrix}
\alpha & 0  \\
0 & 1 \\ 
\end{pmatrix}\\&
              = \begin{pmatrix}
\alpha a( y+\lambda) & 0  \\
0 &c(y+\lambda) \\ 
\end{pmatrix}\begin{pmatrix}
\alpha & 0  \\
0 & 1 \\ 
\end{pmatrix}=\begin{pmatrix}
\alpha^2 a(y+\lambda) & 0  \\
0 & c(y+\lambda) \\
\end{pmatrix}
  \end{align*}
  which is true by assumption. From Proposition~\ref{proposition2.2}
  we conclude that $\mathcal{K}\cong \tilde{\mathcal{K}}$.
\end{proof}

\noindent
For instance, with $a(y)=y$, $c(y)=1+y^2$ , $\lambda=2$ and $\alpha=1$ finds that
\begin{center}
  $g=\begin{pmatrix}
    y & 0 \\
    0 & 1+y^2 \\ 
  \end{pmatrix}$ and $\tilde{g}=\begin{pmatrix}
    y+2 &0 \\
    0 & 5+y^2+4y\\ 
  \end{pmatrix}$
\end{center}
give rise to isomorphic K\"ahler-Poisson algebras.
     
\section{Summary}

\noindent
In this paper, we have started to investigate isomorphism classes of
K\"ahler-Poisson algebras for rational functions in two
variables. Although a complete classification has not been obtained
there are several subclasses of metrics which can be explicitly
described. Proposition~\ref{prop:constant.metric} shows that the class
of constant metrics can be parametrized by one (complex) parameter,
and Proposition~\ref{prop:diagonal.g} describes the isomorphism
classes of diagonal metrics only depending on $x$. Furthermore,
several sufficient conditions are derived
(Proposition~\ref{proposition x} and
Proposition~\ref{prop:diagonal.y}) and a number of examples are given
illustrating that seemingly different metrics may give rise to
isomorphic K\"ahler-Poisson algebras.

\end{document}